\documentclass{article}
\usepackage{amsthm}
\usepackage{amsmath,amssymb}
\usepackage{helvet}
\usepackage{docmute}
\usepackage{tcolorbox}
\usepackage[enableskew]{youngtab}
\usepackage{pxrubrica}
\usepackage{color}
\usepackage{tikz}
\usepackage{ytableau}
\usepackage{braket}
\usepackage{mathtools}
\usepackage{hyperref}
\usepackage{endnotes}
\usepackage{authblk}
\usepackage{float}
\usepackage{cite}
\hypersetup{
 setpagesize=false,
 bookmarks=true,
 bookmarksdepth=tocdepth,
 bookmarksnumbered=true,
 colorlinks=false,
 pdftitle={},
 pdfsubject={},
 pdfauthor={},
 pdfkeywords={}}
\theoremstyle{definition}
\newtheorem{dfn}{Definition}[section]
\newtheorem{thm}{Theorem}[section]
\newtheorem{lem}{Lemma}[section]

\newtheorem{prop}{Proposition}[section]
\newtheorem{cor}{Corollary}[section]
\newtheorem*{main*}{Main result}
\theoremstyle{definition}
\newtheorem{case}{Example}[section]
\newtheorem*{prop*}{proposition}
\numberwithin{equation}{section}

\newtheorem{rem}{Remark}[section]
\theoremstyle{plain}

\makeatletter
\renewcommand{\AB@affilsep}{\quad\protect\Affilfont}
\let\AB@affilsepx\AB@affilsep

\makeatother

\begin{document}

\title{Staircase hook-length ratios and special values of Jacobi polynomials}
\author{Tatsushi Shimazaki\thanks{Department of Mathematics, Graduate School of Science, Kobe University}}
\date{}

\maketitle

\def\thefootnote{\fnsymbol{footnote}}
\footnote[0]{%
\begin{tabular}{@{}ll}
E-mail: & \texttt{tsimazak@math.kobe-u.ac.jp} \\
Keywords: & Jacobi polynomials, Grothendieck polynomials, set-valued tableaux. \\
MSC2020: & 05E05, 05A19, 33C45.
\end{tabular}}
\def\thefootnote{\arabic{footnote}}

\begin{abstract}
We relate hook-length products for adjacent staircase partitions to special values of Jacobi polynomials. 
This connection expresses the number of semistandard tableaux in terms of Jacobi polynomials defined via Gauss hypergeometric functions. 
From this identity, we derive the special values of stable Grothendieck polynomials and $K$-theoretic Schur $P$-functions indexed by adjacent staircase partitions. 
These values provide ratios of the numbers of set-valued and shifted set-valued semistandard tableaux. 
This connection is further clarified by the theory of excited Young diagrams, which characterizes the coefficients in these specializations.
\end{abstract}

\vspace{3mm}

\vspace{2mm}

\section{Introduction}\label{intro}

The hook-length formula, established by Frame, Robinson, and Thrall~\cite{FRT}, is a fundamental result in algebraic combinatorics. 
It provides a product formula for the number of standard Young tableaux. 
This formula admits a number of generalizations, such as $q$-analogues and product formulas for skew shapes~\cite{MPP1, MPP2, MPP3}.
Recent studies extended these results to $K$-theoretic analogues.
Throughout the paper we work with staircase shapes.
Among all partitions, staircase shapes arise from their symmetry and from their role in representation theory and basic hypergeometric series.
Geometrically, they correspond to Schubert varieties associated with the longest element in the Weyl group or maximal isotropic Grassmannians.
Analyzing these shapes leads to an understanding of the singularities of general Schubert varieties and the properties of their structure sheaves.

In $K$-theory, Schur polynomials are generalized to Grothendieck polynomials. 
The stable Grothendieck polynomial represents the structure sheaves of Schubert varieties in flag varieties~\cite{Lascoux,LS}. 
This geometric interpretation links algebraic combinatorics with the $K$-theory of flag varieties. 
The combinatorial aspects of Grothendieck polynomials also relate to integrable systems and vertex models \cite{MotegiSakai}.
To describe the coefficients of these polynomials, Buch~\cite{Buch} introduced set-valued semistandard Young tableaux. 
These tableaux provide the combinatorial expression for the stable Grothendieck polynomial.
A shifted analogue exists for the symplectic and orthogonal types.
Ikeda and Naruse~\cite{IkedaNaruse} introduced $K$-theoretic factorial Schur $P$- and $Q$-functions. 
These polynomials represent the structure sheaves of Schubert varieties in the equivariant $K$-theory of maximal isotropic Grassmannians. 
These functions also possess a combinatorial description in terms of excited Young diagrams~\cite{IkedaNaruse2009, IkedaNaruse}.

In this paper, we show that ratios of hook-length products for staircase partitions are given by specializations of Jacobi polynomials.
Specifically, our main results are as follows:

\begin{enumerate}
    \item We prove the ratio of the number of semistandard tableaux for adjacent staircase partitions is given by a specialization of Jacobi polynomials (Theorem~\ref{thm:main}).
    \item We derive explicit recursive relations for staircase ratios in terms of stable Grothendieck polynomials and $K$-theoretic Schur $P$-functions (Propositions~\ref{stairG} and \ref{stairGP}).
    \item We relate the Jacobi polynomial evaluations to the appearance of powers of $3$ in shifted tableau enumeration via excited Young diagrams (Proposition~\ref{prop:three-power-count}).
\end{enumerate}

Jacobi polynomials yield a uniform expression for the hook-length ratios in both even and odd cases.
This relates the counting problem to classical special functions and the recursive structure in the $K$-theoretic setting.

This paper is organized as follows.
In Section~\ref{pre}, we give the definitions used throughout the paper. 
Section~\ref{main} presents the main result on the relation between hook-length ratios and Jacobi polynomials.
 In Section~\ref{Gro}, we apply the main result to stable Grothendieck polynomials and $K$-theoretic Schur $P$-functions.
Section~\ref{exc} is devoted to the combinatorial interpretation via excited Young diagrams.
Finally, we summarize the results obtained in this paper in Section~\ref{con}.

\section{Preliminaries}\label{pre}

This section provides the basic definitions of partitions, hook lengths, and the special functions appearing in our main results.

\subsection{Partitions, Young diagrams, and hook lengths}

A \emph{partition} of a non-negative integer $l$ is a finite sequence of non-negative integers
$\lambda=(\lambda_1,\lambda_2,\ldots,\lambda_r)$ such that $\lambda_1 \geq \lambda_2 \geq \cdots \geq \lambda_r$ and $\lambda_1+\cdots+\lambda_r=l$.
We identify $(\lambda_1,\dots,\lambda_r)$ with $(\lambda_1,\dots,\lambda_r,0)\ (r \in \mathbb{Z}_{>0})$.
The \emph{length} of $\lambda$, denoted $\ell(\lambda)$, is defined as the smallest integer $r$ such that $\lambda_r > 0$ and $\lambda_{r+1} = 0$.
The \emph{Young diagram} ${\rm YD}_\lambda$ of $\lambda$ is the set
\begin{align*}
{\rm YD}_\lambda := \{ (i,j) \in (\mathbb{Z}_{>0})^2 \mid 1\leq i \leq \ell(\lambda),\ 1 \leq j \leq \lambda_i \}.
\end{align*}
We call an element $(i,j)$ of $\lambda$ a box. 
The Young diagram of $\lambda$ consists of $|\lambda| = \sum_i \lambda_i$ boxes arranged in left-justified rows, where the $i$-th row contains $\lambda_i$ boxes.
 Following the English convention, we label the box in the $i$-th row and $j$-th column by $(i,j)$, where $i$ and $j$ increase downward and to the right, respectively.
In this paper, each partition is regarded as its Young diagram.

Let $\lambda$ and $\nu$ be Young diagrams with $\lambda \supset \nu$. 
The set-theoretic difference $\lambda - \nu$ is called a \emph{skew Young diagram}, denoted $\lambda/\nu$. 
It consists of the boxes in $\lambda$ that are not in $\nu$. 

For each box $u=(i,j) \in \lambda$, the \textit{hook length} $h(u)$ is defined as the number of boxes to its right in the same row, plus the number of boxes below it in the same column, plus one for the box $u$ itself. 
Specifically, if $\lambda'$ denotes the conjugate partition of $\lambda$, then
\begin{align*}
h(u) = (\lambda_i - j) + (\lambda'_j - i) + 1.
\end{align*}
The \textit{hook length product} of $\lambda$ is defined as 
\begin{align*}
H_\lambda = \prod_{u \in \lambda} h(u).
\end{align*}

A particularly symmetric family of partitions is the \textit{staircase partition} $\delta_k = (k-1, k-2, \dots, 1)$ for $k \in \mathbb{Z}_{>0}$. 
Since $\delta_k$ is self-conjugate, its hook length at $u=(i,j) \in \delta_k$ is given by the simple formula:
\begin{align}\label{eq:Hk}
h(u) = 2(k-i-j)+1.
\end{align}
We denote the hook length product of $\delta_k$ by $H_k$. From \eqref{eq:Hk}, it follows that
\begin{align}\label{eq:Hk_formula}
H_k = \prod_{u \in \delta_k} h(u) = \prod_{i=1}^{k-1} \prod_{j=1}^{k-i} \bigl(2(k-i-j)+1\bigr).
\end{align}

\begin{case}
For $\delta_4=(3,2,1)$, the corresponding Young diagram is\vspace{2mm}
\begin{align*}
{\raisebox{-5.5pt}[0pt][0pt]{$\delta_4 =\ $}} {\raisebox{-0pt}[0pt][0pt]{\ytableaushort{\ \ \ ,\ \ ,\ }}}.\\[5mm]
\end{align*}
The hook length product is obtained from~\eqref{eq:Hk_formula}
\begin{align*}
H_4 = 5\times3\times1\times3\times1\times1 = 45.
\end{align*}
\end{case}

\subsection{Gauss hypergeometric function}

The \textit{Gauss hypergeometric function} is defined by the power series
\begin{align}\label{hgf}
{}_{2}F_{1}\!\left( \begin{matrix} a, b \\ c \end{matrix} ; z \right) 
= \sum_{s=0}^\infty \frac{(a)_s (b)_s}{(c)_s s!} z^s,
\end{align}
for $c \notin \mathbb{Z}_{\le0}$. 
Here $(\alpha)_s$ denotes the \textit{Pochhammer symbol}: $(\alpha)_0=1$ and $(\alpha)_s = \alpha(\alpha+1)\cdots(\alpha+s-1)$ for $s\ge1$.
The series converges absolutely for $|z|<1$.
In this paper, we consider its analytic continuation to the complex plane with a cut along $[1,\infty)$.
When at least one of the parameters $a$ or $b$ is a non-positive integer, the series terminates as a polynomial in $z$, and the function is well-defined for all $z \in \mathbb{C}$.
The value at $z=1$ is given by the Gauss summation formula (also known as the Chu--Vandermonde identity):
\begin{align}\label{eq:vandermonde}
{}_{2}F_{1}\!\left( \begin{matrix} a, b \\ c \end{matrix} ; 1 \right) 
= \frac{\Gamma(c)\Gamma(c-a-b)}{\Gamma(c-a)\Gamma(c-b)},
\end{align}
valid whenever the series terminates or $\operatorname{Re}(c-a-b)>0$.

\subsection{Jacobi polynomials}

The \textit{Jacobi polynomial} $P_k^{(\alpha, \beta)}(z)$ is a classical orthogonal polynomial.
For a non-negative integer $k$, it is defined by the hypergeometric representation
\begin{align}\label{eq:jacobi_hgf_def}
P_k^{(\alpha, \beta)}(z)
= \frac{(\alpha+1)_k}{k!}\,
{}_2F_1\!\left(
\begin{matrix}
-k,\; k+\alpha+\beta+1\\
\alpha+1
\end{matrix}
;\frac{1-z}{2}
\right).
\end{align}
Here $z$ denotes the same variable as in the Gauss hypergeometric function.
Standard references for the properties of these polynomials and related hypergeometric functions include \cite{AAR1999, Slater}.
Even if $\alpha+1$ is a non-positive integer, the polynomial is still given by the finite expansion
\begin{align}\label{eq:jacobi_sum_def}
P_k^{(\alpha, \beta)}(z)
= \sum_{r=0}^k
\binom{k+\alpha}{k-r}
\binom{k+\beta}{r}
\left(\frac{z-1}{2}\right)^r
\left(\frac{z+1}{2}\right)^{k-r},
\end{align}
which defines a polynomial for all complex parameters.
Although the condition $\alpha, \beta > -1$ is often assumed for orthogonality, we treat $P_k^{(\alpha, \beta)}(z)$ as a polynomial defined for all $\alpha, \beta \in \mathbb{C}$ via \eqref{eq:jacobi_sum_def}.
This polynomial is classically orthogonal on the interval $z \in [-1,1]$ with respect to an explicit weight function.
It also satisfies the symmetry relation
\begin{align*}
P_k^{(\alpha, \beta)}(-z) = (-1)^k P_k^{(\beta, \alpha)}(z).
\end{align*}
This symmetry shows that the specialization at $z=-1$ may equivalently be interpreted as an evaluation at $z=1$ after interchanging the parameters $(\alpha,\beta)$.

In Section~\ref{main}, we show that the ratio of semistandard Young tableau counts for staircase shapes is characterized by the specialization $P_k^{(\alpha, \beta)}(-1)$.
Note that at $z=-1$, the argument of the hypergeometric representation~\eqref{eq:jacobi_hgf_def} becomes $1$.
This formulation remains valid for non-positive integer parameters via the expansion~\eqref{eq:jacobi_sum_def} or an appropriate limiting procedure.

\section{A hypergeometric identity for semistandard tableaux}\label{main}

Our main theorem connects the ratio of hook length products of adjacent staircase partitions to the Jacobi polynomial. 
In what follows, we use the term \textit{semistandard tableaux} as an abbreviation for semistandard Young tableaux. 
The number of semistandard tableaux of shape $\lambda$ with entries in $[n] \coloneqq \{1, \dots, n\}$ is given by the \textit{hook-content formula}~\cite{Macdonald,Noumi}:
\begin{align}\label{hook}
|{\rm SST}(\lambda, n)| = \prod_{u \in \lambda} \frac{n + \operatorname{ct}(u)}{h(u)},
\end{align}
where the content of each box $u=(i,j)$ is defined by $\operatorname{ct}(u) \coloneqq j-i$.
Thus $|{\rm SST}(\lambda,n)|$ is given in a factorized form and is readily computable.

To prove our main result, we first evaluate the algebraic components required for the ratio of these counts.
Let $C_k(n)$ denote the product of the shifted contents of the boxes in the skew shape $\delta_{k+1}/\delta_k$:
\begin{align}\label{eq:Pk_def}
C_k(n) \coloneqq \prod_{u\in\delta_{k+1}/\delta_k} (n+\operatorname{ct}(u)).
\end{align}
The following lemma is needed to obtain a closed form expression for the ratios appearing in the main theorem.

\begin{lem}\label{lem:algebraic_components}
Let $\delta_k=(k-1,k-2,\dots,1)$ be the staircase partition. For positive integers $k$ and $n$ with $n\ge k$, the following hold.
\begin{enumerate}
\item The ratio of the hook-product for adjacent staircase partitions is
\begin{align}\label{eq:hook_ratio}
\frac{H_k}{H_{k+1}} = \frac{1}{(2k-1)!!}.
\end{align}
\item 
\begin{itemize}
\item If $k=2m$ is even, then
\begin{align}\label{eq:Pk_even}
C_{2m}(n) = 4^m \Bigl(\frac{n+1}{2}-m\Bigr)_m\Bigl(\frac{n+1}{2}\Bigr)_m.
\end{align}
\item If $k=2m+1$ is odd, then
\begin{align}\label{eq:Pk_odd}
C_{2m+1}(n) = n \times 4^m \Bigl(\frac{n}{2}-m\Bigr)_m\Bigl(\frac{n}{2}+1\Bigr)_m.
\end{align}
\end{itemize}
\end{enumerate}
\end{lem}

\begin{proof}
To prove (1), we first characterize the skew shape $\delta_{k+1}/\delta_k$. Since $\delta_k=(k-1,k-2,\dots,1)$ and $\delta_{k+1}=(k,k-1,\dots,1)$, the skew diagram consists of exactly $k$ boxes:
\begin{align*}
\delta_{k+1}/\delta_k = \{(i,k+1-i) \mid i=1,2,\dots,k\}.
\end{align*}
Recall that for $\delta_k$, the hook length at box $u$ is given by \eqref{eq:Hk}. 
Note that the multiset of hook lengths of $\delta_{k+1}$ is the union of those of $\delta_k$ and the hook lengths in the first column of $\delta_{k+1}$. 
These are precisely the odd integers $\{1, 3, 5, \dots, 2k-1\}$. It follows that
\begin{align*}
\frac{H_{k+1}}{H_k} = \prod_{j=1}^k (2j-1) = (2k-1)!!,
\end{align*}
establishing \eqref{eq:hook_ratio}.

For (2), the content of the box $(i,k+1-i) \in \delta_{k+1}/\delta_k$ is $\operatorname{ct}(i,k+1-i) = (k+1-i)-i = k+1-2i$. Thus, by definition \eqref{eq:Pk_def}, we have
\begin{align}\label{eq:Pk_raw}
C_k(n) = \prod_{i=1}^{k} (n+k+1-2i).
\end{align}
Let $c_i = k+1-2i$. As $i$ ranges from $1$ to $k$, the values $\{c_i\}$ run through the integers $\{k-1, k-3, \dots, -(k-1)\}$.

If $k=2m$, the set $\{c_i\}$ consists of the odd integers $\{\pm 1, \pm 3, \dots, \pm(2m-1)\}$, and we obtain
\begin{align*}
C_{2m}(n) = \prod_{t=1}^{m} (n^2-(2t-1)^2) = \prod_{t=1}^{m} (n-(2t-1))(n+(2t-1)).
\end{align*}
The Pochhammer form follows by
\begin{align*}
&\prod_{t=1}^{m}\left(n-(2t-1)\right)
=2^m\prod_{t=1}^{m}\left(\frac{n+1}{2}-t\right)
=2^m\left(\frac{n+1}{2}-m\right)_m,\\
&\prod_{t=1}^{m}\left(n+(2t-1)\right)
=2^m\prod_{t=1}^{m}\left(\frac{n+1}{2}+(t-1)\right)
=2^m\left(\frac{n+1}{2}\right)_m.
\end{align*}
Multiplying these two identities yields~\eqref{eq:Pk_even}.

If $k=2m+1$, then the set of contents $\{c_i\}$ consists of the even integers $\{\pm 2, \pm 4, \dots, \pm 2m\}$ together with $0$. 
It follows that
\begin{align*}
C_{2m+1}(n)
= n \times \prod_{t=1}^{m}\bigl(n^2-(2t)^2\bigr)
= n \times \prod_{t=1}^{m} (n-2t)(n+2t).
\end{align*}
To derive the Pochhammer representation, we have
\begin{align*}
\prod_{t=1}^{m}(n-2t) &= 2^m \prod_{t=1}^{m}\Bigl(\frac{n}{2}-t\Bigr)
= 2^m \Bigl(\frac{n}{2}-m\Bigr)_m, \\
\prod_{t=1}^{m}(n+2t) &= 2^m \prod_{t=1}^{m}\Bigl(\frac{n}{2}+t\Bigr)
= 2^m \Bigl(\frac{n}{2}+1\Bigr)_m.
\end{align*}
Combining these with the central factor $n$ leads to \eqref{eq:Pk_odd}.
\end{proof}

For $a, b, c \in \mathbb{C}$, the Gauss hypergeometric function at $z=1$ is given by~\eqref{eq:vandermonde}.
When $c$ is a non-positive integer $c_0 \in \mathbb{Z}_{\le 0}$, the function is defined by the following regularized limit:
\begin{align}\label{eq:reg_def}
{}_{2}F_{1}\!\left( \begin{matrix} a, b \\ c_0 \end{matrix} ; 1 \right) \coloneqq \lim_{\varepsilon \to 0} {}_{2}F_{1}\!\left( \begin{matrix} a, b \\ c_{0} + \varepsilon \end{matrix} ; 1 \right),
\end{align}
where $0 < |\varepsilon| \ll 1$.

\begin{lem}\label{lem:gauss_regularization}
Let $c_0 \in \mathbb{Z}_{\le 0}$ and $a, b \in \mathbb{C}$. 
The regularized value defined in~\eqref{eq:reg_def} exists and is non-zero if and only if the following two conditions hold:
\begin{enumerate}
    \item $c_0 - a - b \in \mathbb{Z}_{\ge 0}$,
    \item $\{c_0 - a, c_0 - b\} \cap \mathbb{Z}_{\le 0}$ contains exactly one element.
\end{enumerate}
\end{lem}

\begin{proof}
Assume conditions (1) and (2) hold. For $c_\varepsilon \coloneqq c_0 + \varepsilon$, the Gauss summation formula~\eqref{eq:vandermonde} gives
\begin{align}\label{eq:Gauss_eps_ratio_v2}
{}_{2}F_{1}\!\left( \begin{matrix} a, b \\ c_{\varepsilon} \end{matrix} ; 1 \right) = \frac{\Gamma(c_\varepsilon)\Gamma(c_\varepsilon - a - b)}{\Gamma(c_\varepsilon - a)\Gamma(c_\varepsilon - b)}.
\end{align}
By condition (1), $\Gamma(c_\varepsilon - a - b)$ is holomorphic and finite at $\varepsilon = 0$, and hence non-zero.
Condition (2) implies that exactly one factor in the denominator has a simple pole at $\varepsilon = 0$. 
Without loss of generality, assume that this factor is $\Gamma(c_\varepsilon - b)$.
In this case, $\Gamma(c_\varepsilon - a)$ is holomorphic and finite at $\varepsilon = 0$ by the same condition.
Thus the poles in the numerator and denominator cancel, so that the limit \eqref{eq:reg_def} exists and is non-zero.

Conversely, suppose that the limit \eqref{eq:reg_def} is finite and non-zero.
If condition (1) is not satisfied, $\Gamma(c_\varepsilon - a - b)$ is either holomorphic or has a pole.
In either case, the orders of poles in \eqref{eq:Gauss_eps_ratio_v2} do not cancel out.
Therefore the limit either vanishes or diverges, which is a contradiction.
Similarly, suppose that condition (2) is violated.
If both $\Gamma(c_\varepsilon - a)$ and $\Gamma(c_\varepsilon - b)$ are holomorphic, the pole of $\Gamma(c_\varepsilon)$ remains.
If both have poles, the pole of the numerator is over-cancelled.
In this situation, the limit fails to be finite and non-zero.

Finally, we evaluate the limit explicitly. 
Using the expansion $\Gamma(-n + \varepsilon) = \frac{(-1)^n}{n! \varepsilon} + O(1)$ for $n \in \mathbb{Z}_{\ge 0}$ and assuming $c_0 - b = -k \in \mathbb{Z}_{\le 0}$, we obtain
\begin{align*}
\lim_{\varepsilon \to 0} \frac{\Gamma(c_0 + \varepsilon)}{\Gamma(c_0 - b + \varepsilon)} 
= \lim_{\varepsilon \to 0} \frac{\displaystyle \frac{(-1)^{-c_0}}{(-c_0)! \, \varepsilon}}{\displaystyle \frac{(-1)^{k}}{k! \, \varepsilon}} 
= \frac{(-1)^{-c_0-k} k!}{(-c_0)!}.
\end{align*}
Since the remaining factors $\Gamma(c_\varepsilon - a - b)$ and $1/\Gamma(c_\varepsilon - a)$ in \eqref{eq:Gauss_eps_ratio_v2} are holomorphic and non-zero at $\varepsilon = 0$, the limit exists and is non-zero. 
This completes the proof.
\end{proof}

\begin{rem} 
Lemma~\ref{lem:gauss_regularization} follows from the Gauss summation formula \eqref{eq:vandermonde} and the pole structure of the Gamma function. 
The proof is included to keep the presentation self-contained and to clarify the parameter constraints required for the main result.
\end{rem}

\begin{thm} \label{thm:main}
Let $\delta_k=(k-1,k-2,\dots,1)$ be the staircase partition.
For positive integers $k$ and $n$ with $n\ge k$, we have
\begin{align}
\frac{|{\rm SST}(\delta_{k+1},n)|}{|{\rm SST}(\delta_k,n)|} = \frac{2^{2k} (k!)^2}{(2k)!} P_k^{(\alpha, \beta)}(-1),
\end{align}
where $\alpha = \frac{n-k-1}{2}$ and $\beta = \frac{-n-k-1}{2}$.
\end{thm}

\begin{proof}
We first compute the ratio of the number of semistandard tableaux using the hook-content formula~\eqref{hook}:
\begin{align} \label{eq:ratio_initial}
\frac{|{\rm SST}(\delta_{k+1},n)|}{|{\rm SST}(\delta_k,n)|} = \frac{H_k}{H_{k+1}}\times C_k(n).
\end{align}
From Lemma~\ref{lem:algebraic_components}, the ratio of hook-products is $H_k/H_{k+1} = 1/(2k-1)!!$.
By the same lemma, the contents of the skew shape $\delta_{k+1}/\delta_k$ are $\{k-1, k-3, \dots, -(k-1)\}$. The product $C_k(n)$ is given by
\begin{align*}
C_k(n) = \prod_{j=0}^{k-1} (n + k - 1 - 2j).
\end{align*}
Substituting these into \eqref{eq:ratio_initial}, we have
\begin{align*}
\frac{|{\rm SST}(\delta_{k+1},n)|}{|{\rm SST}(\delta_k,n)|} &= \frac{1}{(2k-1)!!} \prod_{j=0}^{k-1} (n + k - 1 - 2j).
\end{align*}
Factoring out $2$ from each of the $k$ terms in the product yields
\begin{align*}
\prod_{j=0}^{k-1} (n + k - 1 - 2j) = 2^k \prod_{j=0}^{k-1} \left( \frac{n+k-1}{2} - j \right) = 2^k (\alpha+1)_k.
\end{align*}
Using the identity $(2k-1)!! = (2k)! / (2^k k!)$, we obtain
\begin{align} \label{eq:comb_final_reduction}
\frac{|{\rm SST}(\delta_{k+1},n)|}{|{\rm SST}(\delta_k,n)|} 
= \frac{2^k k!}{(2k)!} \times 2^k (\alpha+1)_k 
= \frac{2^{2k} k!}{(2k)!} (\alpha+1)_k.
\end{align}

Next, we consider the special value of the Jacobi polynomial~\eqref{eq:jacobi_hgf_def} at $z = -1$.
With $\alpha = \frac{n-k-1}{2}$ and $\beta = \frac{-n-k-1}{2}$, it is straightforward to check that
\begin{align*}
k + \alpha + \beta + 1 = k + \frac{n-k-1}{2} + \frac{-n-k-1}{2} + 1 = 0.
\end{align*}
This condition leads to
\begin{align*}
P_k^{(\alpha, \beta)}(-1) = \frac{(\alpha+1)_k}{k!} \,{}_2F_1 \!\left( \begin{matrix} -k, 0 \\ \alpha+1 \end{matrix} ; 1 \right).
\end{align*}
By Lemma~\ref{lem:gauss_regularization}, the regularized ${}_2F_1$ value is 1.
Therefore we obtain
\begin{align*}
P_k^{(\alpha, \beta)}(-1) = \frac{(\alpha+1)_k}{k!}.
\end{align*}
Substituting $(\alpha+1)_k = k! P_k^{(\alpha, \beta)}(-1)$ into \eqref{eq:comb_final_reduction} yields
\begin{align*}
\frac{|{\rm SST}(\delta_{k+1},n)|}{|{\rm SST}(\delta_k,n)|} 
= \frac{2^{2k} k!}{(2k)!} \times k! P_k^{(\alpha, \beta)}(-1) 
= \frac{2^{2k} (k!)^2}{(2k)!} P_k^{(\alpha, \beta)}(-1).
\end{align*}
This completes the proof.
\end{proof}

\begin{rem}
The Jacobi polynomial representation consolidates the Pochhammer symbols from the hook--content formula. The choice $z=-1$ corresponds to the boundary of the orthogonality interval $[-1,1]$. At this point, Lemma~\ref{lem:gauss_regularization} applies directly.
\end{rem}

\begin{case}\label{case:full-verification}

\noindent
(1)\ We consider the case $k=4$, $n=6$.
In this situation, we treat $\delta_4=(3,2,1)$ and $\delta_5=(4,3,2,1)$.
Using the hook--content formula~\eqref{hook},
we compute the number of semistandard tableaux.
For $\delta_4$, the cells and their contents and hook lengths are as follows:
\begin{align*}
\begin{array}{c|c|c}
u & \operatorname{ct}(u) & h(u) \\ \hline
(1,1) & 0 & 5 \\
(1,2) & 1 & 3 \\
(1,3) & 2 & 1 \\
(2,1) & -1 & 3 \\
(2,2) & 0 & 1 \\
(3,1) & -2 & 1
\end{array}
\end{align*}
Thus we have
\begin{align*}
|{\rm SST}(\delta_4,6)|
=\frac{6}{5}\times\frac{7}{3} \times \frac{8}{1} \times \frac{5}{3} \times \frac{6}{1} \times \frac{4}{1}
=896.
\end{align*}
For $\delta_5$, a similar calculation gives
\begin{align*}
|{\rm SST}(\delta_5,6)|
=\prod_{u\in\delta_5}\frac{6+\operatorname{ct}(u)}{h(u)}
=8064.
\end{align*}
Therefore we obtain
\begin{align*}
\frac{|{\rm SST}(\delta_5,6)|}{|{\rm SST}(\delta_4,6)|}
=\frac{8064}{896}
=9.
\end{align*}
On the other hand, for $k=4$, the prefactor in Theorem \ref{thm:main} reduces to
\begin{align*}
\frac{2^{2k} (k!)^2}{(2k)!} = \frac{2^8 (4!)^2}{8!} = \frac{256 \times 576}{40320} = \frac{147456}{40320} = \frac{128}{35}.
\end{align*}
Additionally, the parameters are
\begin{align*}
\alpha = \frac{6-4-1}{2} = \frac{1}{2}, \quad \beta = \frac{-6-4-1}{2} = -\frac{11}{2}.
\end{align*}
Hence the value of the Jacobi polynomial is then given by
\begin{align*}
P_4^{(1/2, -11/2)}(-1) &= \frac{1}{24}\times \left(\frac{3}{2}\right)_4 \\
&= \frac{1}{24} \times \left( \frac{3}{2} \times \frac{5}{2} \times \frac{7}{2} \times \frac{9}{2} \right) = \frac{315}{128}.
\end{align*}
Combining these leads to
\begin{align*}
\frac{2^8 (4!)^2}{8!}P_4^{(1/2, -11/2)}(-1) = \frac{128}{35} \times \frac{315}{128} = 9.
\end{align*}
This coincides with the combinatorial computation.

\vspace{4mm}
\noindent
(2) We examine the case $k=3$, $n=4$.
Applying~\eqref{hook}, we have
\begin{align*}
|{\rm SST}(\delta_3,4)|=20,
\quad
|{\rm SST}(\delta_4,4)|=64.
\end{align*}
This gives the ratio
\begin{align*}
\frac{|{\rm SST}(\delta_4,4)|}{|{\rm SST}(\delta_3,4)|}
=\frac{64}{20}
=\frac{16}{5}.
\end{align*}
In contrast, the prefactor coefficient for $k=3$ is
\begin{align*}
\frac{2^{2k} (k!)^2}{(2k)!} = \frac{2^6 (3!)^2}{6!} = \frac{64 \times 36}{720} = \frac{16}{5}.
\end{align*}
The parameters are
\begin{align*}
\alpha = \frac{4-3-1}{2} = 0, \quad \beta = \frac{-4-3-1}{2} = -4.
\end{align*}
Under these parameters, the value of the Jacobi polynomial equals
\begin{align*}
P_3^{(0, -4)}(-1) = \frac{(1)_3}{3!} = \frac{1 \times 2 \times 3}{6} = 1.
\end{align*}
Thus we obtain
\begin{align*}
\frac{2^6 (3!)^2}{6!}P_3^{(0, -4)}(-1)=\frac{16}{5} \times 1 = \frac{16}{5},
\end{align*}
which is in agreement with the combinatorial result.
\end{case}

\begin{rem}\label{rem:hypergeometric}
The connection to Jacobi polynomials $P_k^{(\alpha,\beta)}(z)$ provides a natural setting for analyzing the combinatorial properties of staircase shapes.

\begin{itemize}
\item 
The hypergeometric representation provides a unified description for both even and odd parity cases.
These cases exhibit distinct combinatorial behaviors, yet both arise from the same ${}_2F_1$ formulation. Their differences are reduced to the choice of parameters in the Jacobi polynomial.

\item
Staircase ratios coincide with special values of Jacobi polynomials. This connection identifies the properties of ${}_2F_1$ with combinatorial identities. Classical relations for the ${}_2F_1$ series yield symmetries and recurrences, reflecting a bijective structure.
   
\item
The hypergeometric expression of $P_k^{(\alpha,\beta)}(z)$ suggests a $q$-deformation via basic hypergeometric series. Such series arise in the specialization of Macdonald polynomials. Our identities provide the $q=1$ limit for these structures within the theory of symmetric functions.
\end{itemize}
\end{rem}

\section{Applications to Grothendieck polynomials and $K$-theoretic Schur $P$-functions}\label{Gro}

In this section, we apply the results from Section~\ref{main} to $K$-theoretic symmetric functions. We first recall the definitions of stable Grothendieck polynomials and $K$-theoretic Schur $P$-functions in the non-factorial setting, as introduced in \cite{Buch, IkedaNaruse}.

\subsection{Jacobi polynomial recurrences for Grothendieck polynomials via set-valued tableaux}
The following is a set-valued generalization of semistandard tableaux.
Let $T_{i,j}$ represent a nonempty subset of $[n]$. 
\begin{dfn}[\!\!\cite{Buch}]\label{svt}
A \textit{set-valued semistandard tableau of shape $\lambda$} is defined as a map $T : {\rm YD}_\lambda \rightarrow 2^{[n]}$ that satisfies the following conditions:
\begin{align*}
|T_{i,j}|\geq 1,\quad \max{T_{i,j}} \leq \min{T_{i,j+1}},\quad \max{T_{i,j}} < \min{T_{i+1,j}}.
\end{align*}
\end{dfn}
Set-valued semistandard tableaux are simply referred to as set-valued tableaux.
For a Young diagram $\lambda$, the set of set-valued tableaux of shape $\lambda$ is denoted by ${\rm SVT}(\lambda)$.
We use the notation ${\rm SVT}(\lambda,n)$ when it is to specify the number of variables $n$.
We consider only the cases with ${\rm SVT}(\lambda,n)\neq \emptyset$.

\begin{case}\label{exSVT}
Let $\lambda = (2,1)={\scalebox{0.4}{{\raisebox{8pt}[0pt][0pt]{\ytableaushort{\ \ ,\ }}}}}$ and $n = 3$. \vspace{0.2cm}
An example $T \in {\rm SVT}({\scalebox{0.3}{{\raisebox{10pt}[0pt][0pt]{\ytableaushort{\ \ ,\ }}}}},3)$ is given by:\vspace{-3mm}
\begin{align*}
T = \quad {\raisebox{2pt}{\ytableaushort{{1}{1,\!3},{2,\!3}}}}.
\end{align*}
The subset $T_{1,2} = \{1,3\} \subset [3]$ is assigned to the box $(1,2)\in {\scalebox{0.4}{{\raisebox{8pt}[0pt][0pt]{\ytableaushort{\ \ ,\ }}}}}$ in this example. To simplify notation, we abbreviate $\scriptsize{{\raisebox{-4.0pt}[0pt][0pt]{\ytableaushort{{1,\!3}}}}}$ as $\scriptsize{{\raisebox{-4.0pt}[0pt][0pt]{\ytableaushort{{13}}}}}$.
All tableaux in ${\rm SVT}({\scalebox{0.3}{{\raisebox{10pt}[0pt][0pt]{\ytableaushort{\ \ ,\ }}}}},3)$ are enumerated below:
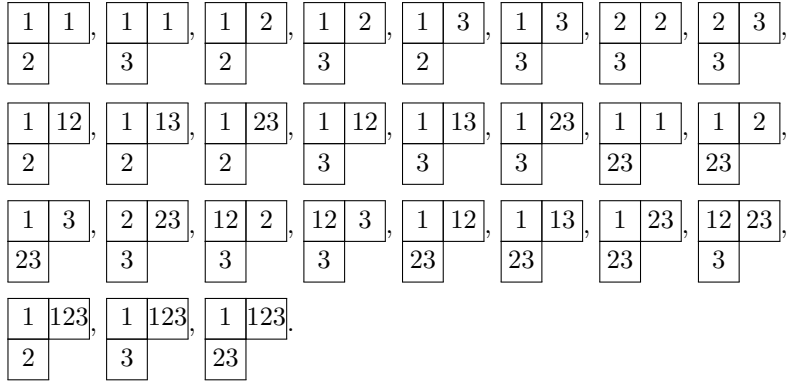
\vspace{-2mm}
\begin{align*}
&{\raisebox{-10pt}[0pt][0pt]{\ytableaushort{11,2}}}{\raisebox{-7pt}[0pt][0pt],\ } {\raisebox{-10pt}[0pt][0pt]{\ytableaushort{11,3}}}{\raisebox{-7pt}[0pt][0pt],\ }{\raisebox{-10pt}[0pt][0pt]{\ytableaushort{12,2}}}{\raisebox{-7pt}[0pt][0pt],\ }{\raisebox{-10pt}[0pt][0pt]{\ytableaushort{12,3}}}{\raisebox{-7pt}[0pt][0pt],\ }{\raisebox{-10pt}[0pt][0pt]{\ytableaushort{13,2}}}{\raisebox{-7pt}[0pt][0pt],\ }{\raisebox{-10pt}[0pt][0pt]{\ytableaushort{13,3}}}{\raisebox{-7pt}[0pt][0pt],\ }{\raisebox{-10pt}[0pt][0pt]{\ytableaushort{22,3}}}{\raisebox{-7pt}[0pt][0pt],\ }{\raisebox{-10pt}[0pt][0pt]{\ytableaushort{23,3}}}{\raisebox{-7pt}[0pt][0pt],\ }\\\\\\
&{\raisebox{-3pt}[0pt][0pt]{\ytableaushort{1{12},2}},\ }{\raisebox{-3pt}[0pt][0pt]{\ytableaushort{1{13},2}},\ }{\raisebox{-3pt}[0pt][0pt]{\ytableaushort{1{23},2}},\ } {\raisebox{-3pt}[0pt][0pt]{\ytableaushort{1{12},{3}}},\ }{\raisebox{-3pt}[0pt][0pt]{\ytableaushort{1{13},3}},\ }{\raisebox{-3pt}[0pt][0pt]{\ytableaushort{1{23},3}},\ }{\raisebox{-3pt}[0pt][0pt]{\ytableaushort{1{1},{23}}},\ }{\raisebox{-3pt}[0pt][0pt]{\ytableaushort{1{2},{23}}},\ }\\\\
&{\raisebox{-10pt}[0pt][0pt]{\ytableaushort{1{3},{23}}}}{\raisebox{-7pt}[0pt][0pt],\ }{\raisebox{-10pt}[0pt][0pt]{\ytableaushort{2{23},{3}}}}{\raisebox{-7pt}[0pt][0pt],\ }{\raisebox{-10pt}[0pt][0pt]{\ytableaushort{{12}2,3}}}{\raisebox{-7pt}[0pt][0pt],\ }{\raisebox{-10pt}[0pt][0pt]{\ytableaushort{{12}3,3}}}{\raisebox{-7pt}[0pt][0pt],\ }{\raisebox{-10pt}[0pt][0pt]{\ytableaushort{{1}{12},{23}}}}{\raisebox{-7pt}[0pt][0pt],\ }{\raisebox{-10pt}[0pt][0pt]{\ytableaushort{1{13},{23}}}}{\raisebox{-7pt}[0pt][0pt],\ }{\raisebox{-10pt}[0pt][0pt]{\ytableaushort{1{23},{23}}}}{\raisebox{-7pt}[0pt][0pt],\ }{\raisebox{-10pt}[0pt][0pt]{\ytableaushort{{12}{23},3}}}{\raisebox{-7pt}[0pt][0pt],\ }\\\\
&{\raisebox{-17pt}[0pt][0pt]{\ytableaushort{1{123},{2}}}}{\raisebox{-14pt}[0pt][0pt],\ }{\raisebox{-17pt}[0pt][0pt]{\ytableaushort{1{123},{3}}}}{\raisebox{-14pt}[0pt][0pt],\ }{\raisebox{-17pt}[0pt][0pt]{\ytableaushort{1{123},{23}}}}{\raisebox{-14pt}[0pt][0pt].\ }\\\\\\
\end{align*}
\end{case}

We define the weight of $T$ as
\begin{align}\label{weight}
\omega(T) = (\omega_1(T),\omega_2(T),\dots,\omega_n(T)) \in (\mathbb{Z}_{\geq 0})^n,
\end{align}
where $\omega_m(T) \coloneqq |\{ m \in [n] \mid m \in T \}|$.
Additionally, we define a monomial of $x=(x_1,x_2,\dots,x_n) \in \mathbb{C}^n$ having a weight $\omega(T)$ as
\begin{align}\label{mono}
x^{\omega(T)} = x_1^{\omega_1(T)}x_2^{\omega_2(T)}\cdots x_n^{\omega_n(T)}.
\end{align}
With these settings, following \cite{Buch}, we define the Grothendieck polynomial $G_\lambda$ as follows:
\begin{align}\label{GG}
G_{\lambda}=G_{\lambda}(x_1,x_2,\dots,x_n\mid\beta) = \sum_{T \in {\rm SVT}(\lambda,n)}\beta^{|T|-|\lambda|}x^{\omega(T)},
\end{align}
where $\beta$ is a parameter and $|T|$ is the number of assigned positive integers in $T$. 
We emphasize that the parameter $\beta$ in~\eqref{GG} is independent of the parameter $\beta$ appearing in the Jacobi polynomial~\eqref{eq:jacobi_hgf_def}.
In Example~\ref{exSVT}, the polynomial $G_{\scalebox{0.2}{{\raisebox{8pt}[0pt][0pt]{\ytableaushort{\ \ ,\ }}}}}(x_1,x_2,x_3 \mid \beta)
= G_{\scalebox{0.2}{{\raisebox{8pt}[0pt][0pt]{\ytableaushort{\ \ ,\ }}}}}$
is given by
\begin{align}\label{exG}
G_{\scalebox{0.2}{{\raisebox{8pt}[0pt][0pt]{\ytableaushort{\ \ ,\ }}}}}
=&\ s_{\scalebox{0.2}{{\raisebox{8pt}[0pt][0pt]{\ytableaushort{\ \ ,\ }}}}}
+ \beta(x_1^2x_2^2 + x_1^2x_3^2 + x_2^2x_3^2 + 3x_1^2x_2x_3 + 3x_1x_2^2x_3 + 3x_1x_2x_3^2) \notag\\
&+ \beta^2(2x_1^2x_2^2x_3 + 2x_1^2x_2x_3^2 + 2x_1x_2^2x_3^2) + \beta^3x_1^2x_2^2x_3^2.
\end{align}
Specializing the parameter $\beta$ to $0$ yields $G_{\lambda}(x \mid 0)=s_{\lambda}(x)$, the Schur polynomial. 
On the other hand, under the specialization $x_i=1$ $(i=1,2,\dots,n)$ and $\beta=1$, we obtain  
\begin{align}\label{numberSVT}
G_{\lambda}(1,1,\dots,1 \mid 1)=|{\rm SVT}(\lambda,n)|.
\end{align}
The number $|{\rm SVT}(\lambda,n)|$ is given the following formula.
\begin{cor}[\!\!\cite{FujiiNobukawaShimazaki1}]
For any Young diagram $\lambda$, we have
\begin{align}\label{exSVT}
&|{\rm SVT}(\lambda, n)|\notag \\
&= \sum_{k_1 = 0}^{0}\sum_{k_2 = 0}^{1} \cdots \sum_{k_{n} = 0}^{n - 1}\binom{0}{k_1}\binom{1}{k_2}\cdots\binom{n - 1}{k_{n}}\prod_{1\leq i < j \leq n}\frac{\lambda_i - \lambda_j +j-i + k_i - k_j}{j-i}.
\end{align}
\end{cor}

The Grothendieck polynomial $G_\lambda$ is related to the enumeration of semistandard tableaux and hypergeometric functions.
To make this connection explicit, the Holman hypergeometric function $F^{(n)}$ is defined as follows.
\begin{dfn}[\!\!\!\!\!\!\cite{Holman}]
For positive integers $n,v,w$ and parameters $(A_{ij}), (a_{ij}), (b_{ij}), (z_{i1})$, the function $F^{(n)}$ is defined by
\begin{align*}
&F^{(n)}((A_{ij})_{(n-1)\times (n-1)}|(a_{ij})_{n \times v}|(b_{ij})_{n \times w}|(z_{i1})_{n\times 1}) \\
&=\!\! \sum_{k_1,\dots,k_n=0}^{\infty}\!\!\left( \prod_{1\leq i<j\leq n}\frac{A_{ij}+k_i-k_j}{A_{ij}} \right)\!\!\!\left( \prod_{j=1}^{v}\prod_{i=1}^n(a_{ij})_{k_i} \right)\!\!\!\left( \prod_{j=1}^{w}\prod_{i=1}^n\frac{1}{(b_{ij})_{k_i}} \right)\!\!\!\left( \prod_{i=1}^n z_{i1}^{k_i} \right)\!\!,
\end{align*}
where $(A_{ij})_{(n-1)\times(n-1)}$ denotes
\begin{align*}
\begin{pmatrix}
  A_{12} & & & \\
  A_{13} & A_{23} & & \text{\Huge 0} \\
  \vdots & \vdots & \ddots & \\
  A_{1n} & A_{2n} & \cdots & A_{n-1,n}
\end{pmatrix}.
\end{align*}
\end{dfn}
This function is introduced as a convenient framework
to express the hook-length ratios in a form that naturally leads
to the Jacobi polynomial.
The function $F^{(n)}$ also appears in the representation theory of the Lie groups $U(n+1)$ and $SU(n+1)$.

The following result makes explicit the relationship
between the Holman hypergeometric function $F^{(n)}$
and the Grothendieck polynomial.
\begin{thm}[\!\!{\cite{FujiiNobukawaShimazaki1}}]\label{thm:holman_identity}
For any Young diagram $\lambda$, we have\vspace{-2mm}
\begin{align}\label{eq:holman_identity}
&G_{\lambda}(1,1,\dots,1\mid \beta) \notag\\
&= |{\rm SST}(\lambda,n)| \times F^{(n)} \left( \begin{pmatrix}
  A_{12} & & & \\
  A_{13} & A_{23} & & \text{\Huge $0$} \\
  \vdots & \vdots & \ddots & \\
  A_{1n} & A_{2n} & \cdots & A_{n-1,n}
\end{pmatrix} \middle| \begin{pmatrix}
  0 \\
  -1 \\
  \vdots \\
  -n+1
\end{pmatrix} \middle| \begin{pmatrix}
  1 \\
  1 \\
  \vdots \\
  1
\end{pmatrix} \middle| \begin{pmatrix}
  -\beta \\
  -\beta \\
  \vdots \\
  -\beta
\end{pmatrix} \right)\!\!,
\end{align}
where $A_{ij} = \lambda_i - \lambda_j + j - i$.
\end{thm}

For brevity, the hypergeometric function on the right-hand side
in~\eqref{eq:holman_identity} is denoted by $F^{(n)}_{\lambda}(-\beta)$.
Combining Theorem~\ref{thm:main} and Theorem~\ref{thm:holman_identity}
yields recurrence relations for the Grothendieck polynomials
of staircase Young diagrams.
The following result gives an explicit expression for the dependence on $n$ in terms of classical special functions.

\begin{prop}\label{stairG}
The ratio of the polynomials $G_{\delta_{k+1}}$ and $G_{\delta_k}$ satisfies the following recurrence relation:
\begin{align} \label{eq:recurrence_jacobi_unified}
\frac{G_{\delta_{k+1}}(1,\dots,1\mid \beta)}{G_{\delta_k}(1,\dots,1\mid \beta)} = \frac{2^{2k} (k!)^2}{(2k)!} P_k^{(\alpha, \beta)}(-1) \times \frac{F^{(n)}_{\delta_{k+1}}(-\beta)}{F^{(n)}_{\delta_k}(-\beta)},
\end{align}
where $\alpha = \frac{n-k-1}{2}$ and $\beta = \frac{-n-k-1}{2}$.
\end{prop}

\begin{proof}
From Theorem 4.1, the ratio of the polynomials for the partitions $\delta_{k+1}$ and $\delta_k$ is given by
\begin{align*}
\frac{G_{\delta_{k+1}}(1,\dots,1\mid \beta)}{G_{\delta_k}(1,\dots,1\mid \beta)} = \frac{|{\rm SST}(\delta_{k+1},n)|}{|{\rm SST}(\delta_k,n)|} \times \frac{F^{(n)}_{\delta_{k+1}}(-\beta)}{F^{(n)}_{\delta_k}(-\beta)}.
\end{align*}
By Theorem 3.1, the ratio of the semistandard tableaux counts is expressed exactly in terms of the Jacobi polynomial at $z=-1$ as
\begin{align*}
\frac{|{\rm SST}(\delta_{k+1},n)|}{|{\rm SST}(\delta_k,n)|} = \frac{2^{2k} (k!)^2}{(2k)!} P_k^{(\alpha, \beta)}(-1).
\end{align*}
Substituting this into the above relation, we obtain \eqref{eq:recurrence_jacobi_unified}. 
\end{proof}

\begin{rem}
The function $F^{(n)}_{\lambda}(-\beta)$ appearing in the recurrence relations is derived from the stable Grothendieck polynomial $G_\lambda(x \mid \beta)$, where $\beta$ is a formal parameter. 
Specifically, this function is a polynomial in $\beta$ whose coefficients are themselves symmetric polynomials in the variables $x_i$. 
Consequently, no analytic issues such as convergence or the choice of domain arise; all identities are understood as formal relations within the polynomial ring in $\beta$ over the ring of symmetric functions.
\end{rem}

\begin{case}\label{exratio}
We consider the same setting for $k$ and $n$ as in Example~\ref{case:full-verification}.

\noindent (1)\ Let $k=4$ and $n=6$.
From~\eqref{exSVT}, the value $|{\rm SVT}(\delta_4, 6)|$ is
\begin{align*}
&|{\rm SVT}(\delta_4, 6)| \notag \\
&= \sum_{k_1 = 0}^{0}\sum_{k_2 = 0}^{1}\sum_{k_3=0}^2\sum_{k_4=0}^3\sum_{k_5=0}^4
   \binom{0}{k_1}\binom{1}{k_2}\binom{2}{k_3}\binom{3}{k_4}\binom{4}{k_5}\\
&\ \ \ \ \times
   \prod_{1\leq i < j \leq 5}\frac{\lambda_i - \lambda_j + j-i + k_i - k_j}{\,j-i} \\
&= 2479329.
\end{align*}
In the same way, the evaluation shows that $|{\rm SVT}(\delta_4,6)|=134865$.  
Thus the ratio becomes
\begin{align*}
\frac{|{\rm SVT}(\delta_5,6)|}{|{\rm SVT}(\delta_4,6)|}
   = \frac{2479329}{134865}
   =\frac{3401}{185}.
\end{align*}
On the other hand, setting $\beta=1$ yields
\begin{align*}
F_{\delta_5}^{(6)}(-1)=\frac{2479329}{8064}, \quad 
F_{\delta_4}^{(6)}(-1)=\frac{134865}{896}.
\end{align*}
Therefore we obtain
\begin{align*}
\frac{2^{8} (4!)^2}{8!} P_4^{(1/2, -11/2)}(-1) \times \frac{F_{\delta_5}^{(6)}(-1)}{F_{\delta_4}^{(6)}(-1)} 
= 9 \times \frac{\frac{2479329}{8064}}{\frac{134865}{896}}
=\frac{3401}{185}.
\end{align*}
This is identical to the combinatorial evaluation.

\vspace{4mm}
\noindent (2)\ For $k=3$ and $n=4$, we have from \eqref{exSVT} that
\begin{align*}
|{\rm SVT}(\delta_4,4)|=729,\quad |{\rm SVT}(\delta_3,4)|=159.
\end{align*}
Hence the ratio is given by
\begin{align*}
\frac{|{\rm SVT}(\delta_4,4)|}{|{\rm SVT}(\delta_3,4)|}
   = \frac{729}{159}
   = \frac{243}{53}.
\end{align*}
Next, substituting $\beta=1$ into \eqref{eq:holman_identity} results in
\begin{align*}
F_{\delta_4}^{(4)}(-1)=\frac{729}{64}, \quad 
F_{\delta_3}^{(4)}(-1)=\frac{159}{20}.
\end{align*}
Thus we obtain
\begin{align*}
\frac{2^6 (3!)^2}{6!} P_3^{(0, -4)}(-1) \times \frac{F_{\delta_4}^{(4)}(-1)}{F_{\delta_3}^{(4)}(-1)} = \frac{16}{5}\times \frac{\tfrac{729}{64}}{\tfrac{159}{20}}
   = \frac{243}{53}.
\end{align*}
This is consistent with the combinatorial evaluation.
\end{case}

In~\cite{FujiiNobukawaShimazaki1}, a summation formula for the Holman hypergeometric function is provided:
\begin{thm}[\!\!{\cite{FujiiNobukawaShimazaki1}}]\label{F(1)}
For any Young diagram $\lambda$, we have
\begin{align*}
F^{(n)}\! \left(\!\! \begin{pmatrix}
  A_{12} &\ & & \\
  A_{13} & A_{23} & & \text{\Huge $0$} \\
  \vdots & \vdots &\ddots & \\
  A_{1n} & A_{2n} &\cdots & A_{n-1,n}
\end{pmatrix} \middle| \begin{pmatrix}
  0 \\
  -1 \\
  \vdots \\
  -n+1
\end{pmatrix} \middle| \begin{pmatrix}
  1 \\
  1 \\
  \vdots \\
  1
\end{pmatrix} \middle| \begin{pmatrix}
  1 \\
  1 \\
  \vdots \\
  1
\end{pmatrix}\!\! \right)=\frac{1}{|{\rm SST}(\lambda,n)|}.
\end{align*}
\end{thm}

The following result follows from combining Theorem 3.1 and Theorem~\ref{F(1)}, since both theorems provide different expressions for the same ratio.
\begin{cor}\label{prop:holman_gauss_relation}
Let $k$ and $n$ be positive integers with $n\ge k$. 
The ratio of the Holman hypergeometric functions for adjacent staircase partitions is expressed as the evaluation of the Jacobi polynomial as follows:
\begin{align} \label{eq:holman_jacobi_relation}
\frac{F^{(n)}_{\delta_{k}}(1)}{F^{(n)}_{\delta_{k+1}}(1)} = \frac{2^{2k} (k!)^2}{(2k)!} P_k^{(\alpha, \beta)}(-1),
\end{align}
where $\alpha = \frac{n-k-1}{2}$ and $\beta = \frac{-n-k-1}{2}$. 
\end{cor}

\subsection{Jacobi polynomial recurrences for $K$-theoretic Schur $P$-functions via shifted set-valued tableaux}
In this paper, we adopt the definition of the $K$-theoretic Schur $P$-functions using shifted set-valued semistandard tableaux.
We deal with the non-factorial version, i.e., the case in which all the extra parameters other than the symmetric variables are set to zero.

A \emph{strict partition} is a sequence of positive integers 
$\mu = (\mu_1, \mu_2, \dots, \mu_r)$ 
satisfying $\mu_1 > \mu_2 > \cdots > \mu_r > 0$. 
The \emph{length} of $\mu$ is denoted by $\ell(\mu) = r$.
Given a strict partition $\mu$, the \emph{shifted Young diagram} of shape $\mu$ is
\begin{align*}
{\rm SYD}_\mu \coloneqq 
\{ (i,j) \in (\mathbb{Z}_{>0})^2 \mid 1\leq i \leq \ell(\mu),\ i \leq j \leq \mu_i+i-1 \}.
\end{align*}
It is obtained by indenting the $i$-th row of the Young diagram of $\lambda$
by $i-1$ boxes to the right.
For example, the shifted Young diagram corresponding to the strict partition $\tilde{\delta}_4 = (3, 2, 1)$ is obtained below:
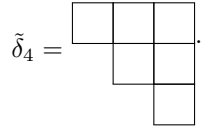
\vspace{3mm}
\begin{align*}
{\raisebox{-5.5pt}[0pt][0pt]{$\tilde{\delta}_4 =\ $}} {\raisebox{-0pt}[0pt][0pt]{\ytableaushort{\ \ \ ,\none\ \ ,\none\none\ }}}.\\[5mm]
\end{align*}

We introduce the definitions of shifted set-valued tableaux based on~\cite{IkedaNaruse}. These tableaux serve as the combinatorial foundation for defining $K$-theoretic Schur $P$-functions.
For $s \in [n]$, let $s' \coloneqq s-\frac{1}{2}$.
We denote $[n']\coloneqq \{1',2',\dots,n'\}$ and $[n',n]\coloneqq \{1',1,2',2,\dots,n',n\}$.
Let $\widetilde{T}_{i,j}$ be a nonempty subset of $[n',n]$. 
\begin{dfn}[\!\!\cite{IkedaNaruse}]\label{ssvt}
A \textit{shifted semistandard set-valued tableau of shape $\mu$} is a map $\widetilde{T} : {\rm SYD}_\lambda \rightarrow 2^{[n',n]}$ satisfying the following conditions:
\begin{enumerate}
\item $|\widetilde{T}_{i,j}| \geq 1$,\quad $\max \widetilde{T}_{i,j} \leq \min \widetilde{T}_{i,j+1}$,\quad $\max \widetilde{T}_{i,j} \le \min \widetilde{T}_{i+1,j}$;
\item $s \in [n]$ appears at most once in each column;
\item $s' \in [n']$ appears at most once in each row;
\item $\widetilde{T}_{i,i} \subseteq [n]$.
\end{enumerate}
\end{dfn}
We abbreviate shifted semistandard set-valued tableaux as shifted set-valued tableaux and denote the set of shifted set-valued tableaux of shape $\mu$ by ${\rm SSVT}_P(\mu)$.
When specifying the number of variables $n$, we use the notation ${\rm SSVT}_P(\mu,n)$. We consider only the cases where ${\rm SSVT}_P(\mu,n) \neq \emptyset$.

The weight of a shifted set-valued tableau $\widetilde{T}$ is defined as $\widetilde{\omega}_s(\widetilde{T}) \coloneqq |\{ s',s \in [n',n] \mid s',s \in \widetilde{T} \}|$. 
The weight $\widetilde{\omega}(\widetilde{T})$ and the corresponding monomial $x^{\tilde{\omega}(\widetilde{T})}$ are defined analogously to those in~\eqref{weight} and~\eqref{mono}.
Following~\cite{IkedaNaruse}, the \textit{$K$-theoretic Schur $P$-functions} are defined as a generating series over these tableaux:
\begin{align}\label{GP}
GP_{\mu}=GP_\mu(x \mid \beta)\coloneqq \sum_{\widetilde{T} \in {\rm SSVT}_P(\mu,n)}\beta^{|\widetilde{T}|-|\mu|}x^{\widetilde{\omega}(\widetilde{T})},
\end{align}
where $|\widetilde{T}|$ represents the total number of positive integers and positive half-integers assigned in $\widetilde{T}$. 
Taking $\beta = 0$ in~\eqref{GP} yields the Schur $P$-function.
The combinatorial properties of these functions and shifted tableaux are developed in \cite{Stembridge1989}.

\begin{case}
Let $\mu = \tilde{\delta}_3 = (2,1)={\scalebox{0.4}{{\raisebox{7pt}[0pt][0pt]{\ytableaushort{\ \ ,\none\ }}}}}$ and $n=3$. All tableaux in ${\rm SSVT}_P({\scalebox{0.3}{{\raisebox{9pt}[0pt][0pt]{\ytableaushort{\ \ ,\none\ }}}}},3)$ are enumerated below:\vspace{3mm}
\begin{align*}
&{\scalebox{1.0}{{\raisebox{1pt}[0pt][0pt]{\ytableaushort{11,\none2}}}}},\ 
{\scalebox{1.0}{{\raisebox{1pt}[0pt][0pt]{\ytableaushort{11,\none3}}}}},\ 
{\scalebox{1.0}{{\raisebox{1pt}[0pt][0pt]{\ytableaushort{12,\none3}}}}},\ 
{\scalebox{1.0}{{\raisebox{1pt}[0pt][0pt]{\ytableaushort{22,\none3}}}}},\ 
{\scalebox{1.0}{{\raisebox{1pt}[0pt][0pt]{\ytableaushort{1{\ \!2'},\none2}}}}},\ 
{\scalebox{1.0}{{\raisebox{1pt}[0pt][0pt]{\ytableaushort{1{\ \!2'},\none3}}}}},\ 
{\scalebox{1.0}{{\raisebox{1pt}[0pt][0pt]{\ytableaushort{1{\ \!3'},\none3}}}}},\ 
{\scalebox{1.0}{{\raisebox{1pt}[0pt][0pt]{\ytableaushort{2{\ \!3'},\none3}}}}},\\\\\\
&{\raisebox{0pt}[0pt][0pt]{\ytableaushort{1{12'},\none2}}},\ {\raisebox{0pt}[0pt][0pt]{\ytableaushort{1{12'},\none3}}},\ {\raisebox{0pt}[0pt][0pt]{\ytableaushort{1{2'2},\none3}}},\ {\raisebox{0pt}[0pt][0pt]{\ytableaushort{1{12},\none{3}}}},\ {\raisebox{0pt}[0pt][0pt]{\ytableaushort{1{13'},\none3}}},\ {\raisebox{0pt}[0pt][0pt]{\ytableaushort{1{23'},\none3}}},{\raisebox{0pt}[0pt][0pt]{\ytableaushort{1{1},\none{23}}}},\ {\raisebox{0pt}[0pt][0pt]{\ytableaushort{1{\ \!2'},\none{23}}}},\\\\\\
&{\raisebox{0pt}[0pt][0pt]{\ytableaushort{1{2'3'},\none{3}}}},\ {\raisebox{0pt}[0pt][0pt]{\ytableaushort{2{23'},\none{3}}}},\ {\raisebox{0pt}[0pt][0pt]{\ytableaushort{{12}2,\none3}}},\ {\raisebox{0pt}[0pt][0pt]{\ytableaushort{{12}{\ \!3'},\none3}}},{\raisebox{0pt}[0pt][0pt]{\ytableaushort{{1}{12'},\none{23}}}},\ {\raisebox{0pt}[0pt][0pt]{\ytableaushort{1{{\scalebox{0.8}{1$2^{\prime}$$3^{\prime}$}}},\none{3}}}},\ {\raisebox{0pt}[0pt][0pt]{\ytableaushort{1{{\scalebox{0.8}{$2^{\prime}$23}}},\none{3}}}},\ {\raisebox{0pt}[0pt][0pt]{\ytableaushort{{12}{23'},\none3}}},\\\\\\
&{\raisebox{0pt}[0pt][0pt]{\ytableaushort{1{{\scalebox{0.8}{1$2^{\prime}$2}}},\none{3}}}},\ {\raisebox{0pt}[0pt][0pt]{\ytableaushort{1{{\scalebox{0.8}{12$3^{\prime}$}}},\none{3}}}},\ {\raisebox{0pt}[0pt][0pt]{\ytableaushort{1{{\scalebox{0.63}{1$2^{\prime}$2$3^{\prime}$}}},\none{3}}}}.\\
\end{align*}
The function $GP_{\scalebox{0.2}{{\raisebox{8pt}[0pt][0pt]{\ytableaushort{\ \ ,\none\ }}}}}(x_1,x_2,x_3 \mid \beta)
= GP_{\scalebox{0.2}{{\raisebox{8pt}[0pt][0pt]{\ytableaushort{\ \ ,\none\ }}}}}$ is
\begin{align}\label{exGP}
GP_{\scalebox{0.2}{{\raisebox{8pt}[0pt][0pt]{\ytableaushort{\ \ ,\none\ }}}}}
=&\ P_{\scalebox{0.2}{{\raisebox{8pt}[0pt][0pt]{\ytableaushort{\ \ ,\none\ }}}}}
+ \beta(x_1^2x_2^2 + x_1^2x_3^2 + x_2^2x_3^2 + 3x_1^2x_2x_3 + 3x_1x_2^2x_3 + 3x_1x_2x_3^2) \notag\\
&+ \beta^2(2x_1^2x_2^2x_3 + 2x_1^2x_2x_3^2 + 2x_1x_2^2x_3^2) + \beta^3x_1^2x_2^2x_3^2.
\end{align}
\end{case}
By comparing~\eqref{exG} and~\eqref{exGP}, we find that the Grothendieck polynomial and the $K$-theoretic Schur $P$-function are equal for staircase partitions and strict partitions.
This extends to the general case, established by Lewis and Marberg:
\begin{thm}[\!\!\cite{LewisMarberg}]\label{prop:LewisMarberg}
For a partition $\lambda$ and a strict partition $\mu$, it follows that $G_\lambda=GP_\mu$ if and only if $\lambda=\delta_k$ and $\mu=\tilde{\delta}_k$.
\end{thm}

This equality implies that the recurrence from Corollary~\ref{stairG} holds for the $K$-theoretic Schur $P$-functions as well, leading to a unified representation via the Jacobi polynomial.

\begin{prop}\label{stairGP}
The $K$-theoretic Schur $P$-functions satisfy the following unified recurrence relation:
\begin{align} \label{eq:GP_recurrence_jacobi}
\frac{GP_{\tilde{\delta}_{k+1}}(1,\dots,1\mid \beta)}{GP_{\tilde{\delta}_k}(1,\dots,1\mid \beta)} = \frac{2^{2k} (k!)^2}{(2k)!} P_k^{(\alpha, \beta)}(-1) \times \frac{F^{(n)}_{\delta_{k+1}}(-\beta)}{F^{(n)}_{\delta_k}(-\beta)},
\end{align}
where $\alpha = \frac{n-k-1}{2}$ and $\beta = \frac{-n-k-1}{2}$.
\end{prop}

\begin{proof}
The statement follows directly from Theorem~\ref{prop:LewisMarberg}. Since $\delta_k$ and $\delta_{k+1}$ are the only partitions satisfying the identity $GP_\mu = G_\lambda$ for the respective strict partitions $\tilde{\delta}_k$ and $\tilde{\delta}_{k+1}$, we have
\begin{align*}
GP_{\tilde{\delta}_k}=G_{\delta_k},\quad GP_{\tilde{\delta}_{k+1}}=G_{\delta_{k+1}}.
\end{align*}
Substituting these identities into the recurrence relation for the $K$-theoretic Grothendieck polynomials given in Proposition~\ref{stairG}, which utilizes the Jacobi polynomial evaluation from Theorem~\ref{thm:main}, we obtain the desired relation.
\end{proof}

\section{A combinatorial interpretation via excited Young diagrams}
\label{exc}

This section provides a combinatorial interpretation of the main formula
in terms of excited Young diagrams, in which the Jacobi polynomial appearing in our results naturally arises.
We adopt the conventions and notation of Ikeda and Naruse~\cite{IkedaNaruse2009}.
Under suitable specializations,
the resulting product expressions reduce to special values
of hypergeometric functions.

\subsection{Representation of Schur and Grothendieck polynomials by excited Young diagrams}
We first recall the representation of Schur polynomials. 
Define
\begin{align*}
D_n := \{ (i,j) \mid 1 \leq i \leq n, i \leq j \} \subseteq (\mathbb{Z}_{\geq 0})^2.
\end{align*}
 An element $(i,j) \in D_n$ is called a box. Let $D$ be a finite subset of $D_n$. An \textit{elementary excitation} is a transformation $D \to D' = (D \setminus \{ (i,j) \}) \cup \{ (i+1,j+1) \}$, provided that $(i, j) \in D$ and none of $(i + 1, j)$, $(i, j + 1)$, or $(i + 1, j + 1)$ belong to $D$. The condition $(i + 1, j) \notin D$ is vacuous for $i=j$.
Let $\mathcal{E}_n(\lambda)$ be the set of all diagrams obtainable from the initial diagram $D_0=D(\lambda)$ by a sequence of elementary excitations.

The \textit{Type A weight} of a diagram $D \in \mathcal{E}_n(\lambda)$ is defined as
\begin{align}\label{weyds}
\operatorname{\operatorname{wt}}(D) := \prod_{(i,j) \in D} x_i.
\end{align}
The weights $x_i$ are assigned to each row $i$, as illustrated below:\\[2mm]
\begin{align*}
{\raisebox{-20pt}[0pt][0pt]{Type A:\quad}}{\raisebox{1pt}[0pt][0pt]{\ytableaushort{{x_1}{x_1}{x_1}{x_1}{\cdots},{x_2}{x_2}{x_2}{x_2}{\cdots},{x_3}{x_3}{x_3}{x_3}{\cdots},{{\raisebox{-2.5pt}[0pt][0pt]\vdots}}{{\raisebox{-2.5pt}[0pt][0pt]\vdots}}{{\raisebox{-2.5pt}[0pt][0pt]\vdots}}{{\raisebox{-2.5pt}[0pt][0pt]\vdots}}{\cdots},{x_n}{x_n}{x_n}{x_n}{\cdots}}}}.\\[15mm]
\end{align*}
The Schur polynomial $s_\lambda(x)$ is the sum of weights over all excited Young diagrams:
\begin{align*}
s_{\lambda}(x) = \sum_{D \in \mathcal{E}_n(\lambda)}\operatorname{wt}(D).
\end{align*}
For instance, consider $\lambda = \delta_3= (2,1)={\scalebox{0.4}{{\raisebox{8pt}[0pt][0pt]{\ytableaushort{\ \ ,\ }}}}}$ and $n = 3$.
We encircle each box in the Young diagram and assign the weight $x_i$ to the boxes in each row $i$ of the set $D$ as follows:
\vspace{3mm}
\begin{align*}
\lambda={\raisebox{1pt}[0pt][0pt]{\ytableaushort{{\color{green!60!black}{\circ}}{\color{red}{\circ}},{\color{blue}{\circ}}}}},\quad {\raisebox{1pt}[0pt][0pt]{\ytableaushort{{x_1}{x_1}{x_1}{x_1}{\cdots},{x_2}{x_2}{x_2}{x_2}{\cdots},{x_3}{x_3}{x_3}{x_3}{\cdots}}}}.\\[3mm]
\end{align*}
We list all excited Young diagrams $D \in \mathcal{E}_3({\scalebox{0.25}{{\raisebox{11pt}[0pt][0pt]{\ytableaushort{\ \ ,\ }}}}})$ as shown below:
\vspace{2mm}
\begin{align}\label{eyds}
\begin{split}
&{\raisebox{1pt}[0pt][0pt]{\ytableaushort{{\color{green!60!black}{\circ}}{\color{red}{\circ}}\ \ ,{\color{blue}{\circ}}\ \ \ ,\ \ \ \ }},\ }\ 
{\raisebox{1pt}[0pt][0pt]{\ytableaushort{{\color{green!60!black}{\circ}}\ \ \ ,{\color{blue}{\circ}}\ {\color{red}{\circ}}\ ,\ \ \ \ }},\ }\ 
{\raisebox{1pt}[0pt][0pt]{\ytableaushort{{\color{green!60!black}{\circ}}\ \ \ ,{\color{blue}{\circ}}\ \  \ ,\ \ \ {\color{red}{\circ}}}},\ }\ 
{\raisebox{1pt}[0pt][0pt]{\ytableaushort{{\color{green!60!black}{\circ}}{\color{red}{\circ}}\ \ ,\ \ \ \ ,\ {\color{blue}{\circ}}\ \ }},}\\[15mm]
&{\raisebox{1pt}[0pt][0pt]{\ytableaushort{{\color{green!60!black}{\circ}}\ \ \ ,\ \ {\color{red}{\circ}}\ ,\ {\color{blue}{\circ}}\ \ }},\ }\ 
{\raisebox{1pt}[0pt][0pt]{\ytableaushort{{\color{green!60!black}{\circ}}\ \ \ ,\ \ \ \ ,\ {\color{blue}{\circ}}\ {\color{red}{\circ}}}},\ }\ 
{\raisebox{1pt}[0pt][0pt]{\ytableaushort{\ \ \ \ ,\ {\color{green!60!black}{\circ}}{\color{red}{\circ}}\ ,\ {\color{blue}{\circ}}\ \ }},\ }\ 
{\raisebox{1pt}[0pt][0pt]{\ytableaushort{\ \ \ \ ,\ {\color{green!60!black}{\circ}}\ \ ,\ {\color{blue}{\circ}}\ {\color{red}{\circ}}}}.}\ \\[10mm]
\end{split}
\end{align}
Summing the weights given by \eqref{weyds} over these excited Young diagrams yields the corresponding Schur polynomial
\begin{align*}
s_{{\scalebox{0.23}{{\raisebox{2pt}[0pt][0pt]{\ytableaushort{\ \ ,\ }}}}}}(x_1,x_2,x_3)&= \sum_{D\in \mathcal{E}_3({\scalebox{0.20}{{\raisebox{10pt}[0pt][0pt]{\ytableaushort{\ \ ,\ }}}}})}\!\!\!\!\operatorname{wt}(D)=x_1^2x_2+x_1x_2^2+\cdots+x_2x_3^2.
\end{align*}
This framework extends to the Grothendieck polynomial.
For a diagram $D$, let $\mathcal{B}(D)$ be the set of boxes $(i,j) \notin D$ that satisfy a specific set of the following conditions with respect to the boxes in $D$:
\begin{align*}
& (1) \quad (i, j) \notin D; \nonumber \\
& (2) \quad \text{ A positive integer } k \text{ exists} \text{ such that } (i + k, j + k) \in D; \nonumber \\
& (3) \quad (i + s, j + s) \notin D \text{ for } 1 \leq s < k; \nonumber \\
& (4) \quad (i + s, j + s + 1) \notin D \text{ for } 0 \leq s < k; \nonumber \\
& (5) \quad (i + s - 1, j + s) \notin D \text{ for } 0 \leq s < k \text{ (this condition is excluded if } i = j). 
\end{align*} 
The weight for $G_\lambda$ is then defined as
\begin{align}\label{weiA}
\operatorname{wt}^{\mathrm{A}}(D) := \prod_{(i,j) \in D} x_i \prod_{(\tilde{i},\tilde{j}) \in \mathcal{B}(D)}(1+\beta x_{\tilde{i}}).
\end{align}
The Grothendieck polynomial $G_{\lambda}(x \mid \beta)$ is given by
\begin{align*}
G_{\lambda}(x \mid \beta) = \sum_{D \in \mathcal{E}_n(\lambda)}\operatorname{wt}^{\mathrm{A}}(D).
\end{align*}
\begin{case}\label{exeydG}
Let $\lambda = \delta_3= (2,1)={\scalebox{0.4}{{\raisebox{8pt}[0pt][0pt]{\ytableaushort{\ \ ,\ }}}}}$ and $n = 3$.
The corresponding excited Young diagrams are listed below:\vspace{3mm}
\begin{align}\label{eydG}
\begin{split}
{\raisebox{1pt}[0pt][0pt]{\ytableaushort{{\color{green!60!black}{\circ}}{\color{red}{\circ}}\ \ ,{\color{blue}{\circ}}\ \ \ ,\ \ \ \ }},\ }\ 
{\raisebox{1pt}[0pt][0pt]{\ytableaushort{{\color{green!60!black}{\circ}}{\color{red}{\beta}}\ \ ,{\color{blue}{\circ}}\ {\color{red}{\circ}}\ ,\ \ \ \ }},\ }\ 
{\raisebox{1pt}[0pt][0pt]{\ytableaushort{{\color{green!60!black}{\circ}}{\color{red}{\beta}}\ \ ,{\color{blue}{\circ}}\ {\color{red}{\beta}}\ ,\ \ \ {\color{red}{\circ}}}},\ }\ 
{\raisebox{1pt}[0pt][0pt]{\ytableaushort{{\color{green!60!black}{\circ}}{\color{red}{\circ}}\ \ ,{\color{blue}{\beta}}\ \ \ ,\ {\color{blue}{\circ}}\ \ }},}\ \\[15mm]
{\raisebox{1pt}[0pt][0pt]{\ytableaushort{{\color{green!60!black}{\circ}}{\color{red}{\beta}}\ \ ,{\color{blue}{\beta}}\ {\color{red}{\circ}}\ ,\ {\color{blue}{\circ}}\ \ }},\ }\ 
{\raisebox{1pt}[0pt][0pt]{\ytableaushort{{\color{green!60!black}{\circ}}{\color{red}{\beta}}\ \ ,{\color{blue}{\beta}}\ {\color{red}{\beta}}\ ,\ {\color{blue}{\circ}}\ {\color{red}{\circ}}}},\ }\ 
{\raisebox{1pt}[0pt][0pt]{\ytableaushort{{\color{green!60!black}{\beta}}\ \ \ ,\ {\color{green!60!black}{\circ}}{\color{red}{\circ}}\ ,\ {\color{blue}{\circ}}\ \ }},\ }\ 
{\raisebox{1pt}[0pt][0pt]{\ytableaushort{{\color{green!60!black}{\beta}}\ \ \ ,\ {\color{green!60!black}{\circ}}{\color{red}{\beta}}\ ,\ {\color{blue}{\circ}}\ {\color{red}{\circ}}}},}\ \\[10mm]
\end{split}
\end{align}
where the boxes marked with the beta $\scalebox{0.7}{{\raisebox{-1pt}[0pt][0pt]{\ytableaushort{{\beta}}}}}$ represent the elements of $\mathcal{B}^{\mathrm{A}}(D)$.
Using this weight~\eqref{weiA}, consider the polynomials corresponding to the excited Young diagrams of~\eqref{eydG}, and taking their sum yields the following Grothendieck polynomial
\begin{align*}
G_{{\scalebox{0.23}{{\raisebox{2pt}[0pt][0pt]{\ytableaushort{\ \ ,\ }}}}}}(x_1,x_2,x_3 \mid \beta) &= \sum_{D\in \mathcal{E}_n^{\mathrm{A}}({\scalebox{0.20}{{\raisebox{10pt}[0pt][0pt]{\ytableaushort{\ \ ,\ }}}}})}\!\!\!\!\operatorname{wt}^{\mathrm{A}}(D) \notag \\
&=x_1^2x_2+x_1x_2^2{\color{red}(1+\beta x_1)}+\cdots+x_2x_3^2{\color{green!60!black}(1+\beta x_1)}{\color{red}(1+\beta x_2)}.
\end{align*}
\end{case}

\begin{rem}
We emphasize that Grothendieck polynomials defined via excited Young diagrams
are well defined for any number of variables $n$, including the case $n<k$.
In particular, the combinatorial definition itself imposes no restriction on $n$.
Throughout this section, additional assumptions on $n$ are introduced only when evaluating specific specializations of weights.
Such assumptions are purely technical and do not affect the validity of the underlying combinatorial definition.
An analogous remark applies to $K$-theoretic Schur $P$-functions associated with shifted shapes,
which will be defined in a later subsection.
\end{rem}

\subsection{Representation of $K$-theoretic Schur $P$-functions by excited Young diagrams}
A similar representation exists for $GP_\mu(x \mid \beta)$~\cite{IkedaNaruse}. 
The set of excited Young diagrams $\mathcal{E}_n(\mu)$ is the same, but the weight function is different. The \textit{Type B weight} of a box $(i,j)$ is defined as
\begin{align*}
  \operatorname{wt}^{\mathrm{B}}(i,j) :=
  \begin{cases}
    x_i & (j=i \text{ or } j>n), \\
    x_i \oplus x_j & (i < j \leq n),
  \end{cases}
\end{align*}
where $x_i \oplus x_j := x_i + x_j + \beta x_ix_j$. 
The assignment of these weights is shown below:
\begin{table}[H]
    \centering
  Type B:\quad
    \label{tab:hogehoge}
    \begin{tabular}{ccccccc}
        \hline
        \multicolumn{1}{|c|}{$\ \ \ \ x_1\ \ \ $} & \multicolumn{1}{|c|}{$x_1\oplus x_2$} & \multicolumn{1}{|c|}{$x_1\oplus x_3$} & \multicolumn{1}{|c|}{$\ \ \cdots\ \ $} & \multicolumn{1}{|c|}{$x_1\oplus x_n$} & \multicolumn{1}{|c|}{$\ \ \ x_1\ \ \ $}  & \multicolumn{1}{|c|}{$\cdots$}\\ \cline{1-7}
         & \multicolumn{1}{|c|}{$x_2$} &  \multicolumn{1}{|c|}{$x_2\oplus x_3$}  & \multicolumn{1}{|c|}{$\cdots$} & \multicolumn{1}{|c|}{$x_2\oplus x_n$} & \multicolumn{1}{|c|}{$x_2$}  & \multicolumn{1}{|c|}{$\ \ \cdots\ \ $} \\ \cline{2-7}
 & & \multicolumn{1}{|c|}{$x_3$} & \multicolumn{1}{|c|}{$\cdots$} & \multicolumn{1}{|c|}{$x_3\oplus x_n$} & \multicolumn{1}{|c|}{$x_3$}  & \multicolumn{1}{|c|}{$\cdots$} \\ \cline{3-7}
 & & & \multicolumn{1}{|c|}{$\ddots$} & \multicolumn{1}{|c|}{$\vdots$} & \multicolumn{1}{|c|}{$\vdots$} & \multicolumn{1}{|c|}{$\vdots$} \\ \cline{4-7}
 & & & & \multicolumn{1}{|c|}{$x_n$} & \multicolumn{1}{|c|}{$x_n$} & \multicolumn{1}{|c|}{$\cdots$} \\ \cline{5-7}
    \end{tabular}.
\end{table}
This defines a weight $\operatorname{wt}^{\mathrm{B}}(D)$ for each diagram $D$, and the $K$-theoretic Schur $P$-function is given by
\begin{align}\label{eydGP}
GP_{\mu}(x \mid \beta) = \sum_{D\in \mathcal{E}_n(\mu)}\operatorname{wt}^{\mathrm{B}}(D).
\end{align}
\begin{case}\label{exeydGP}
Take $\mu = \tilde{\delta}_3= (2,1)={\scalebox{0.4}{{\raisebox{8pt}[0pt][0pt]{\ytableaushort{\ \ ,\none\ }}}}}$ and $n=3$.
The enumeration of excited Young diagrams is presented as shown below:
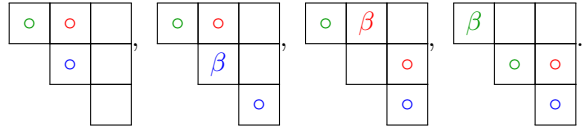
\vspace{3mm}
\begin{align*}
\scalebox{1.0}{{\raisebox{1pt}[0pt][0pt]{\ytableaushort{{\color{green!60!black}{\circ}}{\color{red}{\circ}}\ ,\none{\color{blue}{\circ}}\ ,\none\none\ }},\ }}\ 
\scalebox{1.0}{{\raisebox{1pt}[0pt][0pt]{\ytableaushort{{\color{green!60!black}{\circ}}{\color{red}{\circ}}\ ,\none {\color{blue}\beta}\ ,\none\none{\color{blue}{\circ}}}},\ }}\ 
\scalebox{1.0}{{\raisebox{1pt}[0pt][0pt]{\ytableaushort{{\color{green!60!black}{\circ}}{\color{red}\beta}\ ,\none\ {\color{red}{\circ}},\none\none{\color{blue}{\circ}}}},\ }}\ 
\scalebox{1.0}{{\raisebox{1pt}[0pt][0pt]{\ytableaushort{{\color{green!60!black}\beta}\ \ ,\none{\color{green!60!black}{\circ}}{\color{red}{\circ}},\none\none{\color{blue}{\circ}}}}.\ }}\ \\[5mm]
\end{align*}
The following polynomial is obtained from~\eqref{eydGP}:
\begin{align*}
GP_{{\scalebox{0.23}{{\raisebox{5pt}[0pt][0pt]{\ytableaushort{\ \ ,\none\ }}}}}}(x_1,x_2,x_3 \mid \beta) =& x_1(x_1\oplus x_2)x_2 + x_1(x_1\oplus x_2)x_3{\color{blue}(1+\beta x_2)}\\
&+ x_1(x_2\oplus x_3)x_3{\color{red}(1+\beta(x_1\oplus x_2))}\\
&+x_2(x_2\oplus x_3)x_3{\color{green!60!black}(1+\beta x_1)}.
\end{align*}
\end{case}

Examples~\ref{exeydG} and~\ref{exeydGP} show that the number of excited Young diagrams for $GP_{\tilde{\delta}_k}$ (Type~B) is smaller than the number for $G_{\delta_k}$ (Type~A). This difference reflects a more compact combinatorial structure in the shifted setting.
Under the specialization $x_i=1$ ($1 \le i \le n$) and $\beta=1$, additional combinatorial factors emerge; most notably, powers of $3$ appear in the ratios of the cardinalities of set-valued tableaux. 
For instance, we have
\begin{align*}
\frac{|{\rm SVT}(\delta_4,4)|}{|{\rm SVT}(\delta_3,4)|}
   &= \frac{|{\rm SSVT}_P(\tilde{\delta}_4,4)|}{|{\rm SSVT}_P(\tilde{\delta}_3,4)|}
   = \frac{729}{159}
   = \frac{27 \times 3^3}{53 \times 3},\\
\frac{|{\rm SVT}(\delta_5,6)|}{|{\rm SVT}(\delta_4,6)|}
   &= \frac{|{\rm SSVT}_P(\tilde{\delta}_5,6)|}{|{\rm SSVT}_P(\tilde{\delta}_4,6)|}
   = \frac{2479329}{134865}
   = \frac{3401 \times 3^{6}}{4995 \times 3^3}.
\end{align*}
These observations suggest that the factor $3^{\binom{k-1}{2}}$ comes from the combinatorial structure of $\tilde{\delta}_k$.
Although the total power of $3$ in $|{\rm SSVT}_P(\tilde{\delta}_k, n)|$ depends on $n$, the factor $3^{\binom{k-1}{2}}$ is due to the off-diagonal boxes that are fixed in all excited diagrams. The following result makes this precise.

\begin{prop}\label{prop:three-power-count}
Let $\tilde\delta_k=(k-1,k-2,\dots,1)$ and $n \ge k-1$. 
Under the specialization $x_i=1$ and $\beta=1$, $|{\rm SSVT}_P(\tilde{\delta}_k,n)|$ is divisible by $3^{\binom{k-1}{2}}$.
\end{prop}

\begin{proof}
The shifted staircase shape $\tilde\delta_k$ contains $\binom{k-1}{2}$ off-diagonal boxes $(i,j)$ with $1 \le i < j \le k-1$. 
For each such box, the associated Type~B weight is $x_i \oplus x_j$, which evaluates to $3$ under the specialization $x_i=1$ and $\beta=1$.

In the excited Young diagram formula for the $K$-theoretic Schur $P$-function $GP_{\tilde\delta_k}$, the weight of each excited diagram factors as a product of local contributions. 
In particular, each off-diagonal box contributes a factor depending only on its excitation state. 
The excitation move $(i,j) \to (i+1, j+1)$ preserves the off-diagonal condition
(row index $<$ column index).
Therefore the local weight assigned to such a box remains a multiple of $3$
in every possible excited diagram.
Since this factorization holds for the weight of each excited diagram individually,
each term contributing to $|{\rm SSVT}_P(\tilde{\delta}_k, n)|$
is divisible by $3^{\binom{k-1}{2}}$.
\end{proof}

\begin{rem} \label{rem:rigidity_final}
Combining Theorem~\ref{thm:main} with Proposition~\ref{prop:three-power-count}, we analyze the ratio of the specializations. 
The value $|{\rm SSVT}_P(\tilde{\delta}_k, n)|$ is decomposed by extracting the factor $3^{\binom{k-1}{2}}$ identified in Proposition~\ref{prop:three-power-count}:
\begin{align*}
|{\rm SSVT}_P(\tilde{\delta}_k, n)| = 3^{\binom{k-1}{2}} \times R_k(n),
\end{align*}
where $R_k(n)$ denotes the remaining factor. 
The ratio between adjacent shifted staircase shapes simplifies to
\begin{align*} 
\frac{|{\rm SSVT}_P(\tilde{\delta}_{k+1}, n)|}{|{\rm SSVT}_P(\tilde{\delta}_k, n)|} 
&= \frac{3^{\binom{k}{2}}}{3^{\binom{k-1}{2}}} \times \frac{R_{k+1}(n)}{R_k(n)} \\[1ex]
&= 3^{k-1} \times \frac{R_{k+1}(n)}{R_k(n)}.
\end{align*}
According to the explicit formula in Theorem~\ref{thm:main}, the remaining ratio is evaluated as
\begin{align*}
\frac{R_{k+1}(n)}{R_k(n)} = \frac{2^{2k-2} ((k-1)!)^2}{(2k-2)!} P_{k-1}^{\left(\frac{n-k}{2}, \frac{-n-k}{2}\right)}(-1).
\end{align*}
This shows that the ratio is determined by the factor $3^{k-1}$ corresponding to the growth of off-diagonal boxes and a diagonal contribution expressed via a Jacobi polynomial.
\end{rem}

\begin{rem} \label{rem:Q-functions}
The appearance of powers of $3$ for staircase shapes suggests that similar combinatorial phenomena also arise for $K$-theoretic Schur $Q$-functions.
In this paper, we focus on $K$-theoretic Schur $P$-functions, applying Proposition~\ref{prop:LewisMarberg} to connect $G_\lambda$ and $GP_\mu$ in the staircase setting.
This recursive approach provides a new perspective on the enumerative properties of set-valued tableaux.
For instance, in \cite{FujiiNobukawaShimazaki2, NobukawaShimazaki}, it was proved that the number of set-valued tableaux and shifted set-valued tableaux of any shape is always odd.
The recurrence relations derived here are expected to be useful in studying such enumerative properties for other shapes.
\end{rem}

\section{Conclusion}\label{con}
In this paper, we established the ratio of hook-length products for staircase partitions in terms of special values of the Jacobi polynomial. 
By focusing on the ratios for adjacent staircase partitions, we identified the structural impact of adding a new row to a staircase shape as the difference in the parameters of the Jacobi polynomial. 
This connection was obtained by reducing the explicit calculations based on the hook-content formula to the properties of the Gauss hypergeometric function.

As an application to $K$-theory, we derived recurrence relations for the Grothendieck polynomials and $K$-theoretic Schur $P$-functions. 
These relations admit a combinatorial interpretation through excited Young diagrams. Within this setting, we demonstrated that the invariance of off-diagonal cells and the weight structure of Type B are the determining factors for the power of $3$ appearing in specific specializations.

\section*{Acknowledgments}
The author would like to express the deepest gratitude to Yasuhiko Yamada for continuous support of this work.
The author is also sincerely grateful to Genki Shibukawa for providing the original motivation for this study.
Special thanks are due to Taikei Fujii and Takahiko Nobukawa for their careful reading of the manuscript and for providing insightful comments.
This work was supported by JSPS KAKENHI Grant Number JP22H01116.

\end{document}